\documentclass[12pt]{amsart}
\usepackage{xcolor}
\usepackage[draft]{todonotes} 
\usepackage[english]{babel}
\usepackage{geometry}
\usepackage[T1]{fontenc}
\usepackage[latin1]{inputenc}
\usepackage{amsmath,amssymb,mathrsfs,amsthm, tikz-cd,mathrsfs}
\usepackage{soul}
\usepackage{epsfig}

\usepackage{hyperref} % hypelink 
\hypersetup{
    colorlinks=true,
    linkcolor=blue,
    urlcolor=red,
    pdftitle={},
    }

\usepackage{graphicx}
\usepackage{xypic}
\usepackage{enumitem}
\usepackage{datetime}
\usepackage{xcolor}
\usepackage{fancyhdr}
\newtheorem{theorem}{Theorem}[section]
\newtheorem{proposition}[theorem]{Proposition}

\newtheorem{lemma}[theorem]{Lemma}

\newtheorem{remark}[theorem]{Remark}

\newcommand{\C}{\mathbb{C}}

\newcommand{\Z}{\mathbb{Z}}

\newcommand{\lra}{\longrightarrow}

\newcommand{\R}{\mathbb{R}}

\newcommand{\p}{\mathbb{P}}
\newcommand{\eps}{\varepsilon}

\newcommand{\bei}{\begin{enumerate}[label=(\roman*)]}

\newcommand{\Q}{\mathbb{Q}}
\newcommand{\N}{\mathbb{N}}

\newcommand{\LL}{\mathcal{L}}

\newcommand{\ra}{\rightarrow}

\begin{document}

\vskip 0.5cm 

\title{The  arithmetic volume of hypersurfaces in toric varieties and Mahler measures}

\author{Mounir Hajli}

\address{School of Science, Westlake
University, Hangzhou
310024, China}

\email{hajli@westlake.edu.cn}

%\date{ \today, \currenttime}

\maketitle

\begin{abstract}
In this paper we determine the canonical arithmetic volume of hypersurfaces in smooth  projective toric varieties. As a consequence, we prove a generalized Hodge index theorem on hypersurfaces in  smooth projective  toric varieties.

\end{abstract}
\begin{center}
\end{center}

\begin{center}
\end{center}

\tableofcontents

\section{Introduction}

Let ${Z}$ be    an arithmetic variety over $\mathrm{Spec}(\Z)$, that is  a projective, integral and flat scheme  over $\Z$. 
 Let $d+1$ be the absolute dimension of ${Z}$.     Let $\overline{{\LL}}=(\LL,\|\cdot\|_\phi)$ be a Hermitian  line bundle on ${Z}$,  such that 
the norm $\|\cdot\|_\phi$ is defined by  a continuous  weight  $\phi$.
For any $k\in \N_{\geq 1}$, 
$k\overline{{\LL}}$ denotes $\overline{{\LL}}^{\otimes k}$.\\

 The  arithmetic volume  $\widehat{\mathrm{vol}}(\overline{{\LL}})$  is defined by 
\[
\widehat{\mathrm{vol}}(\overline{{\LL}})=
\underset{k\rightarrow \infty}{\limsup} \frac{\widehat{h}^0(\overline{H^0(Z,k\LL)}_{(\sup,k\phi)})
}{k^{d+1}/(d+1)!},
\]
where $\widehat{h}^0(\overline{H^0(Z,k\LL)}_{(\sup,k\phi)}):=\log \# \{ s\in H^0({Z},k{L}) \mid \|s\|_{\sup,k\phi}\leq 1 \}$.
This arithmetic invariant was introduced by Moriwaki in \cite{Moriwaki2}.

One of the main results of \cite{Moriwaki2} can be stated as follows.
Let $\overline{{\LL}}$ be a  nef $\mathcal{C}^\infty$ Hermitian  line bundle on  $Z$ (see Section \ref{sec3}). Then, the height of $Z$ with respect to $\overline{{\LL}}$ equals to the 
 arithmetic volume of $Z$ with respect to  $\overline{{\LL}}$. Namely
\begin{equation}\label{voldeg}
\widehat{\mathrm{vol}}(\overline{{\LL}})=h_{\overline{{\LL}}}(Z) ,%
\end{equation}
(see \cite{Character} for the definition of the heights). The proof of this result is difficult and relies on  the property of   continuity of
 arithmetic volume function proved in \cite{Moriwaki2}. \\

The explicit determination of the arithmetic volume is a very difficult problem.   When $Z$ is a toric variety, and $\overline{\LL}$ is toric in the sense of Burgos-Philippon-Sombra \cite{Burgos2},  the arithmetic volume of  $Z$ with respect to  $\overline{\LL}$ 
possesses a nice integral representation. Namely the following equation.
\begin{equation}\label{vol-rep}
\widehat{\mathrm{vol}}(\overline \LL)=(d+1)! \int_{\Delta_\LL} 
\max(0,\vartheta_{\overline{{\LL}}})d\mathrm{vol}_M,
\end{equation}
where $M$ is a free $\Z$-module of rank $d$, and
 $\Delta_\LL$ is a rational polytope in $M\otimes_\Z\R$ attached to
$\LL$, $\vartheta_{\overline{{\LL}}}$ a concave function defined in terms of
the metric of $\LL$, and $d\mathrm{vol}_M$ is a normalized 
Lebesgue measure on   $M\otimes_\Z\R$,
see \cite{Burgos3, MounirKJM, Moriwaki}.  We can  prove   \eqref{vol-rep} using     three ingredients. Namely, the following equation 
\[
\widehat{\mathrm{vol}}(\overline \LL)=\underset{k\rightarrow \infty}{\limsup} \frac{
\log \# \{ s\in H^0({Z},k{\LL}) \mid \|s\|_{\mu,k\phi}\leq 1 \}}{k^{d+1}/(d+1)!},
\]
where  $\|\cdot\|_{\mu,k\phi}$ is a Euclidean  norm (see Section \ref{sec2}), the fact that the Euclidean lattice $\overline{H^0(Y,k\LL)}_{(\mu, k\phi)}$ possesses a natural  orthonormal basis with respect to $\|\cdot\|_{\mu,k\phi}$ and some classical results from the geometry of numbers \cite{SouleLattice}. Note that in 
\cite{Burgos3}, the authors bypass the use  of the $L^2$-norm with the fact that the basis of toric sections is orthogonal with respect to the sup-norm for all places, both Archimedean or ultrametric.  

Burgos, Moriwaki, Philippon and Sombra 
\cite{Burgos3} gave a combinatorial proof of
\eqref{voldeg} in the toric setting.\\

Let $M$ be a free $\mathbb{Z}$-module of rank $d$ and $N$ its  dual. We  consider a fan $\Sigma$ on $M_\mathbb{R}=
 M\otimes_\mathbb{Z} \R$ and we denote by $Y
 $  the  associated  toric variety over $\mathbb{Z}$, see for instance \cite{Oda} or \cite[Paragraph 2.2]{Maillot}.   { In the sequel, we assume that $Y$ is   smooth  and projective (this is equivalent to the fact that  $\Sigma$ is nonsingular and the support of $\Sigma$ is  $M_\mathbb{R}$, see \cite[Theorems 1.10 and  1.11]{Oda})}.
   We  set $\mathbb{T}_M:=\mathrm{Hom}_\mathbb{Z}(M,\mathbb{C}^\ast)\simeq (\mathbb{C}^\ast)^d$
  and we denote by $\mathbb{S}_N\simeq (\mathbb{S}^1)^d$ its compact torus. We have an open
   dense immersion  $\mathbb{T}_M\hookrightarrow X$ with an action of $\mathbb{T}_M$ on $Y$ which extends the action of $\mathbb{T}_M$
   on itself by translations. There exists a canonical volume form on $\mathbb{S}_N$  (see \cite[Remarque 6.1.2, and p. 97]{Maillot}) which is denoted by $d\mu_\infty$.
   
   Let $s$ be a nonzero rational function on $Y$.  Mahler measure of $s$ is defined as follows.
\[
m(s)=\int_{\mathbb{S}_N} \log|s|d\mu_\infty.
\] 

Let $\LL$ be an equivariant line bundle on $Y$. It is well known that $\LL$ possesses a natural continuous Hermitian metric which we denote by $\|\cdot\|_{\phi_\infty}$ and  which is given in terms of the combinatorial structure  of $Y$  \cite[Paragraph 3.4]{Maillot}. This metric is called the canonical metric of $\LL$. We  denote this Hermitian line bundle by  $\overline {\LL}_{\phi_\infty}$. 
Let $\p^N$ be the projective space of dimension $N$. The canonical metric of the standard  line bundle $\mathcal O(1)$ on $\p^N$     is given  as follows.
\[
\|\cdot\|_{\phi_\infty}=\frac{|\cdot|}{\max(|x_0|,\ldots,|x_N|)},
\]
where $x_0,x_1,\ldots,x_N$ are the standard homogeneous coordinates in $Y$.
Let $\LL$ be an equivariant line bundle generated by its global sections on $Y$.  Then $\LL$ defines an equivariant morphism $\psi_\LL$ on $Y$ with image in $\p^{h^0(Y, \LL)-1}$.  We can show that
$\overline{\LL}_{{\phi_\infty}}=\psi_\LL^\ast \overline{\mathcal O(1)}_{\phi_\infty}$, see \cite[Paragraph 3.3.3]{Maillot}.  It is worth noting 
 that, up to a positive multiplicative constant, $ \|\cdot\|_{\phi_\infty}$  is the unique toric metric on $\LL$ which has Bernstein-Markov property (see Section \ref{sec2})  with respect to $\mu_\infty$, see Theorem
 \ref{BMtoric}.\\

Let $X$ be a hypersurface in $Y$ which is defined by a
nonzero rational section $s$ of an equivariant line bundle $\mathcal E$ on $Y$. Let
$D$ be an equivariant divisor on $Y$ such that $\mathcal E\simeq \mathcal O(D)$.   Let $s_D$ be the rational function of $Y$ which corresponds to
$s$ by this isomorphism.  The canonical height of the hypersurface $X$ with respect to
$\overline {\LL}_{\phi_\infty}$ is given in terms of Mahler measure of $s_D$. Namely
\[
h_{\overline { \LL}_{\phi_\infty}}(X)=\deg(\LL_\Q)\  m(s_D),
\]
see \cite[Proposition 7.2.1]{Maillot}.  In the sequel, 
$\overline{\mathcal E}_{\psi_\infty}$ denotes the 
line bundle $\mathcal E$ endowed with its canonical metric $\|\cdot\|_{\psi_\infty}$.\\

Our first goal  is to determine the canonical arithmetic volume of  $X$ with respect to
$\overline{\LL}_{{\phi_\infty}}$.
Our first result is  the following equation.
\begin{equation}\label{mt}  %\label{volXY}
\widehat{\mathrm{vol}}_X({\overline {\LL}_{\phi_\infty}})= \mathrm{vol}(\LL_\Q) m(s_D),
\end{equation}

where $\widehat{\mathrm{vol}}_X({\overline {\LL}_{\phi_\infty}})$ is the arithmetic volume of
$X$ with respect to $\overline {\LL}_{\phi_\infty}$ and $ \mathrm{vol}(\LL_\Q)$ is the geometric volume of
$\LL_\Q$, see Theorem \ref{maintheorem}. 
 As a first application, we deduce that 
\begin{equation}\label{volhinfty}
\widehat{\mathrm{vol}}_X({\overline{ \LL}_{\phi_\infty}})=h_{ \overline{\LL}_{\phi_\infty}}({{X}}).
\end{equation}
 In \cite[Proposition 1.5]{MounirKJM}, we proved   that  $\overline{\LL}_{\phi_\infty}$ is arithmetically nef but not big, and hence not arithmetically ample on $Y$ (see Section \ref{sec3} for the definitions of arithmetically nef, big and ample Hermitian line bundles). \\ 

Classically,  \eqref{voldeg} can be used to prove \eqref{volhinfty}. 
Let us outline the proof of this. Let us  first recall  that
 $\overline{\LL}_{\phi_\infty}$ 
 can be approximated by a sequence of 
nef $\mathcal{C}^\infty$ Hermitian  line bundles \cite[Paragraph 2.1]{Zhang}. Then
 use the fact  that the arithmetic volume function and the height are
continuous with respect to the variation of
the metrics \cite[Proposition 3.2.2]{BoGS} and  \cite[Proposition 4.2]{Moriwaki2}. 

Our method for the proof of \eqref{volhinfty} is different and more direct.  As a consequence of this study,  we show that  
the equation

\[
\widehat{\mathrm{vol}}_X({\overline{\LL}_{\phi}})=h_{ \overline{ \LL}_{\phi}}({{X}}),
\]
holds  for every  Hermitian  line ${\overline{ \LL}_{\phi}}$ generated by its small sections on $Y$,  
  see Theorem \ref{ThmVolDeg}.  Thus we  partially recover \eqref{voldeg}. \\

We  proved a generalized Hodge index theorem on toric varieties, see  \cite[Theorem 5.5]{HajliTAMS}. In this paper, we show that the methods of \cite{HajliTAMS} can be used to prove a generalized Hodge index theorem on hypersurfaces in $Y$. 
Let $\phi$ 
be a semipositive weight on $\LL$.  We shall prove that
\[
\widehat{\mathrm{vol}}_X(\overline{\LL}_\phi)\geq h_{\overline{\LL}_\phi}(X),
\]
see Theorem \ref{Hodge}.\\

The approach proposed here for the study of these problems   relies on the 
combinatorial structure of the toric variety
$Y$.  The space $Y(\C)$ possesses a canonical measure which is denoted by
 $\mu_\infty$.  Its restriction to 
 the compact torus of $Y(\C)$ is  a Haar measure. 
To $\mu_\infty$, $\overline{\LL}_{\phi_\infty}$ 
and $\overline{\mathcal E}_{\psi_\infty}$ we attach an Euclidean lattice 
$ \overline{{H^0}\left({{Y}},k \LL+\mathcal E\right)}_{(\mu_\infty,k\phi_\infty+\psi_\infty)}$  for every 
$k\in \N$ (see Section \ref{sec2}). We call 
$ \overline{{H^0}\left({{Y}},k\LL+\mathcal E \right)}_{(\mu_\infty,k\phi_\infty+\psi_\infty)}$ the 
\textit{canonical Euclidean lattice} associated with $\mu_\infty$, $k\overline{\LL}_{\phi_\infty}$ and 
$\overline{\mathcal E}_{\psi_\infty}$.  This lattice plays a central role in this paper.

The Euclidean lattice 
$ \overline{{H^0}\left({{Y}},k\LL+\mathcal E \right)}_{(\mu_\infty,k\phi_\infty+\psi_\infty)}$ induces a structure of  Euclidean lattice on 
${H^0}\left({{X}},(k\LL+\mathcal E)_{|_X} \right)$, which we denote by 
$ \overline{{H^0}\left({{X}},(k\LL+\mathcal E)_{|_X} \right)}_{\mathrm{sq},(\mu_\infty,k\phi_\infty+\psi_\infty)}$ (see  Section \ref{sec2} for more details on the construction). We are naturally led to study the following limits
\[
 \limsup_{k\rightarrow \infty}\frac{{\hat h^0}(\overline{{H^0}({{X}},(k\LL+ \mathcal{E})_{|_{{X}}}) }_{{\mathrm{sq}},(\mu_\infty,k{{\phi_\infty+\psi_\infty}})} )}{k^d/d!},
\]
and
\[
\limsup_{k\rightarrow \infty}\frac{\widehat{\deg}(\overline{{H^0}({{X}},(k\LL+ \mathcal{E})_{|_{{X}}}) }_{ \mathrm{sq}, (\mu_\infty,k\phi_\infty+\psi_\infty)}  )}{k^d/d!}.
\]
In Proposition \ref{4.3} we prove that  
\[
\limsup_{k\rightarrow \infty} \frac{{{{\hat{h}^0}}} \left(\overline{{H^0}({{X}},(k\LL+\mathcal E)_{|_{{X}}}) }_{{\mathrm{sq}}, (\mu_\infty, k\phi_\infty+\psi_\infty)}\right) }{k^d/d!}=\limsup_{k\rightarrow \infty} \frac{{{{\hat{h}^0}}} \left(\overline{{H^0}({{X}},k\LL_{|_{{X}}}) }_{{\mathrm{sq}}, (\mu_\infty, k\phi_\infty)}\right) }{k^d/d!}.
\]
An important point to note here is that 
 the $\chi$-arithmetic volume of canonical Euclidean lattices is zero.
We use  the additivity 
of the  $\chi$-arithmetic degree on admissible metrized sequences, 
and a theorem due to  Szeg\"o and generalized by Deninger, to  deduce the following inequality
\[
\widehat{\mathrm{vol}}_X({\overline {\LL}_{\phi_\infty}})\geq \mathrm{vol}(\LL_\Q)
m(s_D).
\]

Let   $(\|\cdot\|_{\phi_p})_{p=1,2,\ldots}$ be  a sequence of smooth Hermitian  metrics on $\LL$ converging uniformly to $\|\cdot\|_{\phi_\infty}$.

We shall show that
\[
\lim_{p\rightarrow \infty}  \limsup_{k\rightarrow \infty}\frac{{\hat h^0}(\overline{{H^0}({{X}},{k\LL}_{|_{{X}}}) }_{{\mathrm{sq}},(\mu,k{{\phi_p}})} )}{k^d/d!} =\limsup_{k\rightarrow \infty}\frac{{\hat h^0}(\overline{{H^0}({{X}},{k\LL}_{|_{{X}}}) }_{{\mathrm{sq}},(\mu_\infty,k{{\phi_\infty}})} )}{k^d/d!},
\]
where $\mu$ is any smooth probability measure on $Y$ (see Proposition \ref{limit}). Using Bernstein-Markov's property (see Section \ref{sec2})  we shall deduce that
\[
\lim_{p\rightarrow \infty}  \widehat{\mathrm{vol}}_X(\overline{\LL}_{\phi_p})=\limsup_{k\rightarrow \infty}\frac{{\hat h^0}(\overline{{H^0}({{X}},{k\LL}_{|_{{X}}}) }_{{\mathrm{sq}},(\mu_\infty,k{{\phi_\infty}})} )}{k^d/d!}.
\]
It is not difficult to prove   that 
\[
\lim_{p\rightarrow \infty}  \widehat{\mathrm{vol}}_X(\overline{\LL}_{\phi_p})=  \widehat{\mathrm{vol}}_X(\overline{\LL}_{\phi_\infty}).
\]
Using the technical lemma \ref{h1} we should deduce the following
\[
\lim_{k\rightarrow \infty}\frac{\widehat{\deg}(\overline{{H^0}({{X}},(k\LL+ \mathcal{E})_{|_{{X}}}) }_{ \mathrm{sq}, (\mu_\infty,k\phi_\infty+\psi_\infty)}  )}{k^d/d!}= \limsup_{k\rightarrow \infty}\frac{{\hat h^0}(\overline{{H^0}({{X}},(k\LL+ \mathcal{E})_{|_{{X}}}) }_{{\mathrm{sq}},(\mu_\infty,k{{\phi_\infty+\psi_\infty}})})}{k^d/d!}.
\]

Gathering all these computations, we shall conclude the proof of
\eqref{mt}.

\section{Preliminaries}\label{sec2}

A normed $\Z$-module $\overline M=(M,\|\cdot\|)$ is  a $\Z$-module of finite type
endowed with a norm $\|\cdot\|$ on the $\C$-vector space
$M_\C=M\otimes_\Z \C$.  Let $M_{\mathrm{tors}}$ denote  the torsion-module 
of $M$, $M_{\mathrm{free}}=M/M_{\mathrm{tors}}$, and 
$M_\R=M\otimes_\Z\R$. We let $B=\{m\in M_\R: \|m\|\leq 1\}$. There exists
a unique Haar measure  on  $M_\R$ such that the volume of $B$ is $1$. We let
\[
\hat{\chi}(M,\|\cdot\|)=\log \# M_{\mathrm{tors}}-\log \mathrm{vol}(M_\R/ (M/ M_{\mathrm{tors}}) ).
\]

Equivalently, we have
\[
\hat{\chi}(M,\|\cdot\|)=\log \# M_{\mathrm{tors}}-\log\left( \frac{\mathrm{vol}(M_\R/(M/M_{\mathrm{tor}} ))}{\mathrm{vol}(B(M,\|\cdot\|))}\right),
\]
for any choice of a Haar measure of
$M_\R$. \\

The arithmetic degree of $(M,\|\cdot\|)$ is defined as follows

\[
\widehat{\deg}(M,\|\cdot\|)=\widehat{\deg} \overline M=  \hat{\chi}(\overline M)-\hat{\chi}(\overline{\Z}^r),
\]
where $\hat{\chi}(\overline{\Z}^r)=-\log \left( \Gamma(\frac{r}{2}+1)\pi^{-\frac{r}{2}}  \right)$, with $r$ is the rank of $M\otimes_\Z\Q$.  \\

When the norm $\|\cdot\|$ is induced by a Hermitian  product $(\cdot,\cdot)$,
we say that $\overline M$ is an 
 Euclidean lattice. In this situation, we have 
\[
\widehat{\deg}(\overline M)=\log \# M/(s_1,\ldots,s_r)-\log \sqrt{\det((s_i,s_j))_{1\leq i,j\leq r}  },
\]
where $s_1,\ldots,s_r$ are elements of $M$ such that their images 
in $M_\Q$ form a basis. \\

We define $\widehat{H}^0(\overline M)$ and  $\widehat{h}^0(\overline M)$ 
to be
\[
\widehat{H}^0(\overline M)=\left\{ m\in M : \|m\|\leq 1  \right\}\quad\text{and}\quad \widehat{h}^0(\overline M)=\log \# \widehat{H}^0(\overline M).
\]
We let
\[
\widehat{H}^1(\overline M):=\widehat{H}^0(\overline M^\vee)
\quad\text{and}\quad \widehat{h}^1(\overline M):=\widehat{h}^0(\overline M^\vee),
\]
where $\overline M^\vee$ is the $\Z$-module $M^\vee=\mathrm{Hom}_\Z(M,\Z)$ endowed with the dual norm $\|\cdot\|^\vee$ defined as follows
\[
\|f\|^\vee=\sup_{x\in M_\R\setminus\{0\}}\frac{ |f(x)|}{\|x\|}, \quad\forall f\in M^\vee.
\]

Gillet and Soul\'e \cite{SouleLattice} proved   the following \begin{equation}\label{h0h1}
-\log(6) \ \mathrm{rank}\ M \leq  \widehat{h}^0(\overline M)-\widehat{\deg}(\overline M)-\widehat{h}^1(\overline M)\leq \log(\tfrac{3}{2}) \mathrm{rank}\, M+2\log ((\mathrm{rank}\ M)!),
\end{equation}
see also \cite[Proposition 2.1]{Moriwaki2}.\\

A short exact  sequence of Euclidean lattices 
\[
0\longrightarrow \overline{N} \overset{i}{\longrightarrow} \overline{M} \overset{\pi}{\longrightarrow } \overline{Q}\longrightarrow 0,
\]
is said to be admissible if $i_\R$ and the transpose 
of $\pi_\R$ are isometries with respect to the Euclidean 
norms on $N_\R, M_\R$ and $Q_\R$ defining the Euclidean
lattices $\overline N, \overline M$ and $\overline Q$ (for more details see \cite{BostTheta}).\\

We denote by $\|\cdot\|_{{\mathrm{sq}}}$ the norm on $Q$ induced by $\overline M$. It is given by
\begin{equation}\label{sqdef}
\|v\|_{{\mathrm{sq}}}:=\inf_{\substack{m\in M_\R,\\
 \pi_\R(m)=v}} \|m\|,\quad\forall v\in Q_\R.
\end{equation}

Let $Y$ be an arithmetic variety   over $\mathrm{Spec}(\Z)$  of absolute dimension $d+1$. We assume that $Y_\Q$  is smooth. Let $\LL$ be a line bundle on $Y$.\\

 A weight $\phi$ on $\LL(\C)$ is a locally integrable function on the complement of the zero-section in the total space of the dual line bundle $\mathcal{L}^{-1}(\C)$ satisfying the log-homogeneity property 
\[
\phi(\lambda v)=\log|\lambda|+\phi(v)
\]
for all non-zero $v\in \mathcal{L}^{-1}(\C) $ and $\lambda\in \C$.  Let $\phi$ be a weight function on $
\LL$. $\phi$ 
defines a Hermitian  metric  on $\LL$, which we denote by $\|\cdot\|_\phi$. We denote by
$\overline{\LL}_\phi$ the line bundle $\LL$ 
endowed with the metric $\|\cdot\|_\phi$.\\

Let $\mu$ be a probability measure on $Y(\C)$. Let 
$\phi$ (resp. $\psi$) be a  continuous weight function on $\LL$ (resp. $\mathcal E$). Let $k$ be a positive integer.
We endow the space of global sections $H^0(Y,{k\LL+\mathcal E})\otimes_\Z\C $ with the $L^2$-norm given as follows
\[
\|s\|_{(\mu, k\phi+\psi)}:=\left(\int_{Y(\C)} \|s(x)\|_{k \phi+\psi}^2\mu\right)^{\frac{1}{2}} \quad \forall \ s\in H^0(Y,{k\LL+\mathcal E}) \otimes_\Z\C.
\]
Let $(\cdot, \cdot)_{(\mu,k \phi+\psi)}$ denote the associated inner product. 
Also we consider the sup-norm defined by
\[
\|s\|_{\sup,  k\phi+\psi}:=\sup_{x\in Y(\C)}\|s(x)\|_{k\phi+\psi}\quad \forall \ s\in H^0(Y,{k\LL+\mathcal E}) \otimes_\Z\C
\]
 Let $X$ be a subvariety of $Y$. We let
\[
\|s\|_{\sup,  (k\phi+\psi)_{|_X} }:=\sup_{x\in X(\C)}\|s(x)\|_{(k\phi+\psi)_{|_X}}\quad \forall \ s\in H^0(X,{(k\LL+\mathcal E)}_{|_X}) \otimes_\Z \C
\]
 where $(k\phi+\psi)_{|_X}$ 
denotes the weight of the restriction of  $\|\cdot\|_{k\phi+\psi}$ 
to $(k\LL+\mathcal E)_{|_X}$.\\

The Bergman distortion function $\rho(\mu,\overline{\LL})$ is by definition the function given at a point $x\in X$ by
%\begin{equation}\label{distortion}
\[
\rho(\mu,\phi)(x):=\sup_{s\in H^0(X,\LL)_\C\setminus\{0\}}\frac{\|s(x)\|_{\phi}^2}{\quad\|s\|_{(\mu, \phi)}^2}.
\]
%\end{equation}
If $\{s_1,\ldots,s_N\}$ is a $(\mu, \phi)$-orthonormal basis of $H^0(X,{\LL})_\C$, where
$N=\dim_\C H^0(X,\LL)_\C$, then it is
 well known that
 \[
 \rho(\mu,\phi)(x)=\sum_{j=1}^N \|s_j(x)\|_{\phi}^2\quad \forall\,x\in X,
 \]
see \cite[p. 357]{BermanBoucksom}.\\

 We say that  $\mu$ has the Bernstein-Markov property with respect to  $\|\cdot\|_{\phi}$ if  for all $\eps>0$ we have
 \[
 \sup_X \rho(\mu,k\phi)^\frac{1}{2}=O(e^{k\eps}).
 \]

\begin{remark}
 If $\mu$ is a smooth positive volume form and $\|\cdot\|_{\phi}$ is a continuous metric on $\LL$ then
 $\mu$ has the Bernstein-Markov property with respect to $\|\cdot\|_{\phi}$  (see \cite[Lemma 3.2]{BermanBoucksom}).\\
 \end{remark}

The following result provides a new characterization of canonical metrics on equivariant line bundles on toric varieties. 
\begin{theorem}\label{BMtoric} Let $Y_\C$ be nonsingular complex projective variety. Let $L$ be an equivariant line bundle on $Y_\C$ generated by its global sections. Let $\|\cdot\|_\phi$ be a toric Hermitian metric on $L$. Then 
\bei
\item $\|\cdot\|_{\phi_\infty}$  has Bernstein-Markov property with respect to
$\mu_\infty$.

\item $\|\cdot\|_\phi$ has 
Bernstein-Markov property with respect to
$\mu_\infty$ if and only if $\|\cdot\|_\phi=\lambda\|\cdot\|_{\phi_\infty}$ where $\lambda$ is a positive constant.
\end{enumerate}
\end{theorem}

\begin{proof} 

\bei
\item 
Let $k\in \N_{\geq 1}$. 
It is clear that $ \|\chi^m\|_{\sup,k\phi_\infty}=1$ and $ \|\chi^m\|_{\mu_\infty,k\phi_\infty}=1$ for every $m\in k\Delta_L\cap M$. Using this, it is not difficult to see that\[
\|s\|_{\sup,k\phi_\infty}\leq \sum_{m\in k\Delta_L\cap M} |a_m| \leq \sqrt{\#( k\Delta_L\cap M)} \|s\|_{\mu,k\phi}\quad \forall s\in H^0(Y,kL)_\C
\]
where the complex coefficients  $a_m$ are   such that $s=\sum_{m\in k\Delta_L\cap M} a_m \chi^m$. So we have proved (i).

\item 
Let us assume that for every $\eps>0$ we have
\[
\|s\|_{\sup,k\phi}\leq C e^{k\eps} \|s\|_{\mu_\infty,k\phi} \quad \forall k\in \N
\] 
for every $s\in H^0(Z,kL)$, where $C$ is a positive constant. \\

For $s=\chi^m$ with $m\in k\Delta_L\cap M$, we get
\begin{equation}\label{kg0}
e^{-k \check{g}(\frac{m}{k}) } \leq C e^{k\eps} e^{k g(0)}.
\end{equation}
where $g(u):=\log \|s_L\|(e^{-u})$ for every $u\in N_\R$ with $s_L$ is an equivariant rational section of $L$ and
$\check{g}$ is the Legendre-Fenchel transform of $g$, see \cite{Burgos2} for the definition of Legendre-Fenchel transform.
   We deduce from \eqref{kg0} the following \[
\check{g}(x)\geq -g(0)-\eps\quad\forall x\in \Delta_L.
\]
Observe that $\check{g}(x)\leq -g(0)$. 
It follows that
\[
\check{g}(x)=-g(0)\quad\forall x\in \Delta_L.
\]
By \cite[Corollary 12.2.1]{convex}, we infer that $g=g_\infty-g(0)$ where $g_\infty(u)=\log\|s_L\|_{\phi_\infty}(e^{-u})$.
\end{enumerate}

\end{proof}

\section{Arithmetic volume  of hypersurfaces in projective toric varieties}
\label{sec3}

For $X$ an irreducible hypersurface of $Y$, the arithmetic
volume of $X$ with respect to ${\overline{\LL}_\phi}_{|_X}$ is denoted  by $\widehat{\mathrm{vol}}_X( \overline{\LL}_\phi)$ or $\widehat{\mathrm{vol}}_X(\LL,  \|\cdot\|_\phi)$. In other words,
\[
\widehat{\mathrm{vol}}_X( \overline{\LL}_\phi):=\limsup_{k\rightarrow \infty} \frac{1}{k^d/d!} \hat{h}^0(\overline{H^0(X,k\LL_{|_X} )}_{(\sup,k\phi_{|_X})}).
\]

Unless otherwise stated  we assume that  $Y$ is a smooth projective toric variety, and  $\LL$ is an equivariant line bundle generated by its global sections on $Y$.
Let $\mathcal E$ be a  line bundle on $Y$ such that the defining equation
of $X$ is given by a global section $s$ of $\mathcal E$.

Note that the following sequence is exact.
\[
0\ra {H^0({{Y}},
k\LL)}\overset{i}\ra {H^0({{Y}},
{k\LL+\mathcal E})} \overset{\pi}\ra {H^0({{X}},(k\LL+\mathcal E)_{|_{{X}}})} \ra 0\ \ (\text{for}\  k=1,2,\ldots)
\] 
where $i$ is the multiplication map by $s$. \\

Let $k\geq 1$. We consider the following  admissible  exact sequences.
\begin{equation}\label{exact}
0\ra\overline{{H^0({{Y}},
k\LL)}}_{(\mu,k\phi, s)} \overset{i}\ra\overline{H^0({{Y}},
{k\LL+\mathcal E})}_{(\mu,k\phi)}\overset{\pi} \ra \overline{H^0({{X}},(k\LL+\mathcal E)_{|_{{X}}})}_{{\mathrm{sq}},(\mu,k\phi)} \ra 0,
\end{equation}
and
\begin{equation}\label{exactsup}
0\ra \overline{{H^0({{Y}},
k\LL)}}_{(\sup,k\phi, s)} \overset{i}\ra \overline{H^0({{Y}},
{k\LL+\mathcal E})}_{(\sup,k\phi)}\overset{\pi} \ra \overline{H^0({{X}},(k\LL+\mathcal E)_{|_{{X}}})}_{{\mathrm{sq}},(\sup,k\phi)} \ra 0,
\end{equation}
where the metrics of $ \overline{H^0({{Y}},k\LL{{{}}})}_{(\sup,k\phi,s)}$ and $\overline{H^0({{X}},(k\LL+\mathcal E)_{|_{{X}}})}_{{\mathrm{sq}},(\sup,k\phi,s)}$ (resp. $ \overline{H^0({{Y}},k\LL{{{}}})}_{(\mu,k\phi,s)}$ and $\overline{H^0({{X}},(k\LL+\mathcal E)_{|_{{X}}})}_{{\mathrm{sq}},(\mu,k\phi,s)}$  ) are induced by the norm considered on $\overline{H^0({{Y}},
{k\LL+\mathcal E})}_{(\sup,k\phi)}$ (resp. $\overline{H^0({{Y}},
{k\LL+\mathcal E})}_{(\mu,k\phi)}$).

\begin{theorem}\label{thm4.3}
Let $\phi$ be a weight on $\LL$. We assume that
$\|\cdot\|_\phi$ is smooth and  positive. Let $\mu$ be a smooth
probability measure on $Y(\C)$. We have
\bei

\item 

\begin{equation}\label{Th2.2i}
\limsup_{k\rightarrow \infty} \frac{{ \hat h^0 }
 \left(\overline{{H^0}({{X}},(k\LL+\mathcal E)_{|_{{X}}} )}_{{\mathrm{sq}}, (\mu, k\phi) } \right) }{k^d/d!}=\limsup_{k\rightarrow \infty} \frac{{ \hat h^0} \left(\overline{{H^0}({{X}},(k\LL+\mathcal E)_{|_{{X}}}) }_{{\mathrm{sq}}, (\sup, k\phi) } \right) }{k^d/d!}.
\end{equation}
\item

\begin{equation}\label{Th2.2ii}
\limsup_{k\rightarrow \infty} \frac{{{{\hat{h}^0}}} \left(\overline{{H^0}({{X}}, k\LL_{|_{{X}}}) }_{{\mathrm{sq}}, (\mu, k\phi)}\right) }{k^d/d!}=
\widehat{\mathrm{vol}}_X(\overline{\LL}_\phi).
\end{equation}

\end{enumerate}

\end{theorem}

\begin{proof}

Let $\phi$ be a weight on $\LL$ such that
the metric $\|\cdot\|_\phi$ is smooth and positive. 
By \cite[Theorem B, (2.7.3)]{Ran}, we know that for every
$\eps>0$ and $k=1,2,\ldots$

\begin{equation}\label{18}
\|s\|_{\mathrm{sq},(\sup,k\phi+\psi)}\leq C  e^{k\eps} \|s\|_{\sup,(k\phi+\psi)_{|_X}},\quad \forall s\in H^0(X,(k\LL+\mathcal E)_{|_X})\otimes_\Z\C,
\end{equation}
where $C$ is a positive constant depending only on 
$\phi$ and $\mu.$

It is clear that
\begin{equation}\label{19}
\|s\|_{\sup,\phi_{|_{X}}}\leq \|s\|_{\mathrm{sq},(\sup,\phi)}\quad \forall s\in H^0(X,(k\LL+\mathcal E)_{|_X})\otimes_\Z\C
\end{equation}

Combining \eqref{18} and \eqref{19}, and following the proof  \cite[Lemma 2.1]{Moriwaki}), it follows immediately that

%\begin{equation}\label{voltheta}
\[
\limsup_{k\rightarrow \infty} \frac{{{{\hat{h}^0}}}(\overline{{H^0}({{X}},(k\LL+\mathcal E)_{|_{{X}}}) })_{ (\sup,(k\phi+\psi)_{|_X}) } )}{k^d/d!}= \limsup_{k\rightarrow \infty} \frac{{{{\hat{h}^0}}}(\overline{{H^0}({{X}},(k\LL+\mathcal E)_{|_{{X}}}) })_{\mathrm{sq}, (\sup,k\phi+\psi) } )}{k^d/d!}.
\]
%\end{equation}

By Gromov's inequality, there exists a constant $C'$ such that  for every $\eps>0$ and $k\in \N$,

\[
 \|s\|_{(\mu,k\phi+\psi)} \leq  \|s\|_{(\sup, k\phi+\psi)} \leq   C' e^{k\eps} \|s\|_{(\mu,k\phi+\psi)} \quad \forall \ s\in H^0(Y,k\LL+\mathcal E)\otimes_\Z\C.
\]
Hence
\[
 \|s\|_{{\mathrm{sq}}, (\mu,k\phi+\psi)} \leq  \|s\|_{{\mathrm{sq}}, (\sup, k\phi+\psi)} \leq   C' e^{k\eps} \|s\|_{{\mathrm{sq}}, (\mu,k\phi+\psi)} 
 \quad \forall \ s\in H^0({{X}},(k\LL+\mathcal E)_{|_{{X}}}).
\] So  we deduce
(i). The proof (ii)  follows from (i).

\end{proof}

Let $Z$ be an arithmetic variety over $\mathrm{Spec}(\Z)$ of dimension $N+1$.
According to \cite{Moriwaki1},  there are three kinds of positivity of $\overline{{\LL}}=({{\LL}},\|\cdot\|)$ a Hermitian  line bundle on $Z$.
\begin{itemize}
\item \textit{ample} : $\overline{{\LL}}$ is ample if ${{L}}$ is ample on ${Z}$, 
the first Chern form $c_1({\overline{\LL}})$ 
is positive on ${Z}(\C)$ and, for a sufficiently large integer $k$, 
$H^0({Z}, k {\overline{\LL}})$ is generated by the 
set 
\[
\{s\in H^0({Z},k{\overline{\LL}})\mid \|s\|_{\sup}<1 \},
\]
as a $\Z$-module.
\item \textit{nef} : ${\overline{\LL}}$ 
is nef if the first Chern form $c_1( {{\overline{\LL}}})$ is semipositive and 
$\widehat{\deg}({\overline{\LL}}_{|_\Gamma})\geq 0$ for any $1$-dimensional closed subscheme $\Gamma$ in ${Z}$.

\item \textit{big} : ${\overline{\LL}}$ is big 
if  ${\overline{\LL}}_\Q$ is big on ${Z}_\Q$  and 
there is a positive integer $k$ and a 
non-zero section $s$ of $H^0({Z},k{\overline{\LL}})$ 
with
$\|s\|_{\sup}<1$.
  
\end{itemize}

In   the notation of \cite{Moriwaki1} we have  
\begin{equation}\label{kN1}
\hat{h}^1(H^0(Z,k\overline{{\LL}}), \|\cdot\|_{\sup}^{k\overline{{\LL}}} )=o(k^{N+1}), \quad (k\rightarrow \infty),
\end{equation}
for every
    ample Hermitian  line bundle  $\overline{{\LL}}$ on $Z$,
see \cite[p. 428]{Moriwaki2}.\\

The following lemma can be regarded as a slight generalization of
\eqref{kN1}.

\begin{lemma}\label{h1} Let $\mu$ be a  smooth positive volume form  on  $Y$ . Let $\LL$ be an equivariant line bundle generated by its global sections on $Y$. 
Let $\|\cdot\|_{\phi}$ be a continuous  Hermitian  metric on $\LL$ such that
$\|\chi^m\|_{\sup,\phi}\leq 1$ for  every $m\in \Delta_\LL\cap M$. 

Let $X$ be an irreducible  hypersurface of $Y$. With the notations of the previous section,
we have
\[
\widehat h^1(\overline{{H^0}({{X}},{(k\LL+\mathcal E)}_{|_{{X}}}) }_{{\mathrm{sq}},(\sup,k\phi+\psi_\infty)})= o(k^d), \quad (k\rightarrow \infty),
\]
and
\[
\widehat h^1(\overline{{H^0}({{X}},{(k\LL+\mathcal E)}_{|_{{X}}}) }_{{\mathrm{sq}},(\mu,k\phi+\psi_\infty)})= o(k^d), \quad (k\rightarrow \infty).
\]

\end{lemma}

\begin{remark}\label{remark3.2}
Lemma \ref{h1} can be applied to 
$\overline{\LL}_{\phi_\infty}$,  endowed with its 
 canonical  metric.
\end{remark}

%%%%%%%%

%%%%%%%

\begin{proof}[Proof of Lemma \ref{h1}]

To shorten notation,  we write
$\|\cdot\|$ and $\|\cdot\|_{{\mathrm{sq}}}$ instead of $\|\cdot\|_{\sup}$ and 
$\|\cdot\|_{_{{\mathrm{sq}},(\sup,k\phi +\psi_\infty)}}$ respectively. The proof
of the second assertion can be deduced from the first one by using
Bernstein-Markov's property.\\

We let $e_m:=\chi^{m}$ for every $m\in \Delta_{k\LL+\mathcal E}$.\\

Let $k\geq 1$.  Let $\gamma \in \widehat H^1(\overline{{H^0}({{X}},(k\LL+\mathcal E)_{|_{{X}}}) }_{{\mathrm{sq}}})$.    We have
\begin{equation}\label{quotient}
|\gamma(\pi(e_m))|\leq \|\pi(e_m)\|_{{\mathrm{sq}}}\leq \|e_m\|\leq  1, \ \forall m\in (k\Delta_{\LL}+\Delta_{\mathcal E})\cap M.\\
\end{equation}

\
Note that \[
\pi^\ast:H^0({{X}},(k\LL+\mathcal E)_{|_{{X}}})^\vee \longrightarrow  H^0(Y, k\LL+\mathcal E)^\vee\]  is injective. We consider $\pi^\ast (\gamma)\in H^0(Y, k\LL+\mathcal E)^\vee$. There exists a sequence of integers $(a_m)_{m\in (k\Delta_\LL+\Delta_{\mathcal E})\cap M}$ such that 
\[
\gamma \circ \pi=\sum_{m \in \N^{N+1}_k}a_m e_m^\vee,
\]
where $\{e_m^\vee\}_{m\in \Delta_{k\LL+\mathcal E}\cap M}$ denote the dual basis of $\{e_m\}_{m\in \Delta_{k\LL+\mathcal E}\cap M}$. 
From \eqref{quotient} we see that
\[
a_{m}\in \{-1,0,1\}, \quad \forall \;m\in \Delta_{k\LL+\mathcal E}\cap M.
\]

\

Let  $f$ be a rational function which defines $X$.
We have
\[
(\gamma\circ\pi)(f {\chi^{\mu}})=0\quad \forall \mu\in \Delta_{k\LL}\cap M.
\]
So
\[
\begin{split}
0=&\sum_{\nu\in \Delta_{k\LL+\mathcal E}\cap M} a_\nu e_\nu^{\vee}(f{x^{\mu}})\\
=&\sum_{\nu\in \Delta_{k\LL+\mathcal E}\cap M} a_\nu \sum_{m\in \Delta_{\mathcal E}} b_m e_\nu^\vee(\chi^{m}\chi^{\mu})\\
=&
\sum_{\nu\in \Delta_{k\LL+\mathcal E}\cap M} a_\nu \sum_{m\in \Delta_{\mathcal E}} b_m e_\nu^\vee(e_{m+\mu}).
\end{split}
\]
Hence
\[
0=\sum_{\nu\in \Delta_{k\LL+\mathcal E}\cap M} a_\nu b_{\nu-\mu},\quad 
\forall \mu \in \Delta_{k\LL}\cap M,
\]
where we have made the convention that  $b_{\nu-\mu}=0$ whenever 
$\nu-\mu\notin  \Delta_{k\LL}\cap M$.\\

Let us consider the matrix 
\[
C_k= (c_{\mu,m})_{\substack{\mu  \in \Delta_{k\LL}\cap M,\\
m\in \Delta_{k\LL+\mathcal E}\cap M }}\]
 where $c_{\mu,m}=b_{m-\mu}$ for  any
$\mu\in \Delta_{k\LL}\cap M$ and $ m\in \Delta_{k\LL+\mathcal E}\cap M$. So  $C_k$  is a
$h^0(Y,k\LL)\times h^0(Y,k\LL+\mathcal E)$-matrix, where its  $\mu$-row   is given in terms of the coefficients of 
$f{\chi^{\mu}}$.\\

We claim that the rank of $C_k$ is $h^0(Y,k\LL)$. Indeed, let $y=(y_{m})_{m\in \Delta_{k\LL}\cap M} \in \R^{h^0(Y,k\LL)}$. 
By basic linear algebra, we observe that
$C_k^t y=0$ (where $C_k^t$ is the transpose of $C_k$) if and only if $f \sum_{m\in \Delta_{k\LL}\cap M} y_m \chi^m=0$.\\ 

It follows that
\[
\dim \ker C_k=h^0(Y,k\LL+\mathcal E)- h^0(Y,k\LL)=o(k^{d}), 
\]
as $k\rightarrow \infty$. \\

Note that  $(a_m)_{m\in  \Delta_{k\LL+\mathcal E}\cap M}\in \ker  C_k$ 
and recall that $a_m\in\{-1,0,1\}$, so  we can conclude that

\[
\# \widehat H^1(\overline{{H^0}({{X}},(k\LL+\mathcal E)_{|_{{X}}}) }_{{\mathrm{sq}}})=3^{o(k^{d})}.
\]

\end{proof}

%%%%%%%

%%%%%%%

\section{Canonical arithmetic volume  of hypersurfaces}\label{section 4}
Assume that $\LL$ is generated by its global sections on $Y$. 
Let $(\phi_p)_{p=1,2,\ldots}$ be  the sequence of continuous weights on $\LL$ given as follows
\[
\|s(x)\|_{\phi_p}=\frac{|s(x)|}{\left(\sum_{v\in \Delta_\LL\cap M} |\chi^v(x)|^p \right)^{\frac{1}{p}}},\quad p=1,2,\ldots
\]
for every local section $s$ of $\LL$.  

It is well-known that the sequence
$(\|\cdot\|_{\phi_p})_{p=1,2,\ldots}$ converges uniformly to
$\|\cdot\|_{\phi_\infty}$.\\

From now on, we assume moreover that the probability measure  $\mu$ is invariant under
the action of the compact torus of $Y(\C)$.

\begin{proposition}\label{limit}

Let $\LL$ be an equivariant line bundle generated by its global sections. 
Under the above notations and assumptions, we have

\begin{equation}\label{1706}
\lim_{p\rightarrow \infty}  \limsup_{k\rightarrow \infty}\frac{{\hat h^0}(\overline{{H^0}({{X}},{k\LL}_{|_{{X}}}) }_{{\mathrm{sq}},(\mu,k{{\phi_p}})} )}{k^d/d!} =\limsup_{k\rightarrow \infty}\frac{{\hat h^0}(\overline{{H^0}({{X}},{k\LL}_{|_{{X}}}) }_{{\mathrm{sq}},(\mu_\infty,k{{\phi_\infty}})} )}{k^d/d!}.
\end{equation}

\end{proposition}

\begin{proof}

%Since $\LL$
There exists an equivariant map

\[
\psi_\LL: Y\lra \p^{r_\LL},\ x\mapsto \psi_\LL(x)=(\chi^m(x))_{m\in \Delta_\LL\cap M}
\]
where $r_\LL:=\# (\Delta_\LL\cap M)-1$.  We have
\[
\|\cdot\|_{\overline{\LL}_{\phi_\infty}}=\psi_\LL^\ast \|\cdot\|_{\overline{\mathcal O(1)}_{\phi_\infty}}.
\]

Let $\delta\in [0,1]$.  Let $v_0\in \Delta_\LL\cap M$. For every $v\in \Delta_\LL\cap M$, we let
\[
E_{v,\delta}:=\left\{x\in Y(\C): \chi^{v_0}(x)\neq 0\; 
,\;  \delta \tfrac{|\chi^v(x)|}{|\chi^{v_0}(x)|}\leq \tfrac{|\chi^{v'}(x)|}{|\chi^{v_0}(x)|}\leq 
\tfrac{|\chi^v(x)|}{|\chi^{v_0}(x)|}\;\text{for every} \; v' \in (\Delta_\LL\cap M) \right\}.
\]

It is clear that
\[
Y(\C)\setminus \mathrm{div}(\chi^{v_0})=\bigcup_{v\in \Delta_\LL\cap M} E_{v,0}.
\]

For $0<\delta<1$, and for every $p=1,2,\ldots$, $k=1,2,\ldots$, and
$m\in k \Delta_\LL\cap M$,
\[
\begin{split}
( \chi^m,\chi^m)_{(\mu,k \phi_p)}
\geq  &\sum_{v\in \Delta_\LL\cap M} \int_{E_{v,\delta}} \frac{|\chi^m(x)|^2}{ (\sum_{v\in \Delta_\LL\cap M }|\chi^v(x)|^p)^{\frac{2k}{p}}}\mu\geq \frac{\delta^{2k}}{(r_\LL+1)^{\frac{2k}{p}}} I_\delta, \end{split}
\]
where we have put $I_\delta:=\sum_{v\in \Delta_\LL\cap M} \int_{E_{v,\delta}} \mu.$ 

That is
\begin{equation}\label{1011}
( \chi^m,\chi^m)_{(\mu, k {{\phi_p}}) }\geq  \frac{\delta^{2k}}{(r_\LL+1)^{\frac{2k}{p}}} I_\delta.
\end{equation}

%%% 

On one hand,  by noticing that the metrics are invariant under the action of
the compact group $\mathcal{S}$, it is easy to check  that \eqref{1011} gives the following 
\begin{equation}\label{inequality1}
( s,s)_{(\mu, k {{\phi_p}}) }\geq  \frac{\delta^{2k} I_\delta}{(r_\LL+1)^{\frac{2k}{p}}}  ( s,s)_{(\mu_\infty, k {{\phi_\infty}}) }, \quad \forall\  s\in H^0(Y,k\LL)\otimes_\Z \C.
\end{equation}
On the other hand, we have
\begin{equation}\label{inequality2}
(s, s)_{(\mu, k\phi_p)}\leq 
(s, s)_{(\mu_\infty, k\phi_\infty)}, \quad \forall\  s\in H^0(Y,k\LL)\otimes_\Z \C.
\end{equation}
In order to see this, let
$s=\sum_{m  \in k\Delta_\LL\cap M} c_m \chi^m$  be an element of  $H^0(Y,k\LL)\otimes_\Z \C$. By the  invariance of the metrics, we obtain that
\[
\begin{split}
(s, s)_{(\mu,k\phi_p)}=&\sum_{m\in k\Delta_\LL\cap M} |c_m|^2 \int_{Y(\C)}  
\|\chi^m\|_{k\phi_p}^2\mu\\
&\leq \sum_{m\in k\Delta_\LL\cap M} |c_m|^2\\
=& (s, s)_{(\mu_\infty,k\phi_\infty)},
\end{split}
\]
where we have used the fact that  $\|\chi^m\|_{k\phi_p }\leq\|\chi^m\|_{k\phi_\infty }\leq 1$.\\

So, we have proved the following.
\[
\frac{\delta^{2k} I_\delta}{(r_\LL+1)^{\frac{2k}{p}}}  ( s,s)_{(\mu_\infty, k {{\phi_\infty}}) }\leq (s, s)_{(\mu,k\phi_p)}\leq  (s, s)_{(\mu_\infty,k\phi_\infty)}\quad\forall\ 
s\in H^0(Y,k\LL)\otimes_\Z \C.
\]

That is
\[
 \frac{\delta^{k} I_\delta^{\frac{1}{2}}}{(r_\LL+1)^{\frac{k}{p}}} \|\cdot \|_{{\mathrm{sq}}, (\mu_\infty, k {{\phi_\infty}})}\leq \|\cdot \|_{{\mathrm{sq}}, (\mu, k {{\phi_p}})}\leq   
 \|\cdot \|_{{\mathrm{sq}}, (\mu_\infty, k {{\phi_\infty}})} \quad \forall \ k\in \N.\]

From these inequalities, we infer that
\[
\begin{split}
\limsup_{k\rightarrow \infty}\frac{{\hat h^0}(\overline{{H^0}({{X}},
k\LL_{|_{{X}}}) )}_{{\mathrm{sq}},(\mu_\infty,k{{\phi_\infty}})}}{k^d/d!} &\leq \limsup_{k\rightarrow \infty}\frac{{\hat h^0}(\overline{{H^0}({{X}},k\LL_{|_{{X}}})) }_{{\mathrm{sq}},(\mu,k{{\phi_p}})}}{k^d/d!} \\
&\leq 
\limsup_{k\rightarrow \infty}\frac{{\hat h^0}(\overline{{H^0}({{X}},k\LL_{|_{{X}}}) )}_{{\mathrm{sq}},(\mu_\infty,k{{\phi_p}})}}{k^d/d!}\\
&-\log \frac{\delta^2}{(r_\LL+1)^{\frac{2}{p}}}.
\end{split}
\]

By letting $\delta\rightarrow 1^-$, we obtain
\begin{equation}\label{1111}
\left|  \limsup_{k\rightarrow \infty}\frac{{\hat h^0}(\overline{{H^0}({{X}},k\LL_{|_{{X}}}) }_{{\mathrm{sq}},(\mu_\infty,k{{\phi_\infty}})}}{k^d/d!} -\limsup_{k\rightarrow \infty}\frac{{\hat h^0}(\overline{{H^0}({{X}},k\LL_{|_{{X}}}) }_{{\mathrm{sq}},(\mu,k{{\phi_p}})}}{k^d/d!}\right| \leq  
\tfrac{2}{p}\log (r_\LL+1).
\end{equation}

\end{proof}

%%%%

\begin{lemma}\label{Gamma}
\[
\lim_{k\rightarrow \infty} \frac{1}{k^d} \left(\hat \chi(\overline{\Z}^{\#(\Delta_{k\LL+\mathcal E}\cap M)})- \hat \chi(\overline{\Z}^{\# (k\Delta_\LL\cap M)})\right)=0.
\]
\end{lemma}
\begin{proof} 

This is a consequence of Stirling's asymptotic formula.
\end{proof}

%%%%%%%

\begin{proposition}\label{4.3}

\begin{equation}\label{4.3-1}
\limsup_{k\rightarrow \infty} \frac{{{{\hat{h}^0}}} \left(\overline{{H^0}({{X}},(k\LL+\mathcal E)_{|_{{X}}}) }_{{\mathrm{sq}}, (\mu_\infty, k\phi_\infty+\psi_\infty)}\right) }{k^d/d!}=\limsup_{k\rightarrow \infty} \frac{{{{\hat{h}^0}}} \left(\overline{{H^0}({{X}},k\LL_{|_{{X}}}) }_{{\mathrm{sq}}, (\mu_\infty, k\phi_\infty)}\right) }{k^d/d!}.
\end{equation}

\end{proposition}

\begin{proof}

Let $v$ be a global section of $\mathcal E$ which does not
vanish on the compact torus $\mathcal S$ of $Y$. We denote by $V$ the hypersurface defined by
$v$.   We can show  that the following sequence is exact. 
\[
0\longrightarrow {H^0({{X}},
k\LL_{|_X})}\overset{i_V}\longrightarrow {H^0({{X}},
{(k\LL+\mathcal E)_{|_X}})} \overset{\pi_V}\longrightarrow {H^0({{V}},(k\LL+\mathcal E)_{|_{{V}}})} \longrightarrow 0,
\] 
where $i_V$ is the multiplication map by $v$ and
$\pi_V$ is the natural projection map.

Let us consider the following admissible metrized exact sequence

\begin{equation}\label{XV}
\begin{split}
0\longrightarrow \overline{H^0({{X}},
k\LL_{|_X})}_{\mathrm{sq}, (\mu_\infty,k\phi_\infty, v)}\overset{i_V}\longrightarrow & \overline{H^0({{X}},
{(k\LL+\mathcal E)_{|_X}})}_{\mathrm{sq}, (\mu_\infty,k\phi_\infty+\psi_\infty)} \\
& \overset{\pi_V} \longrightarrow \overline{H^0({{V}},(k\LL+\mathcal E)_{|_{{V}}})}_{sq, (\mathrm{sq}, (\mu_\infty,k\phi_\infty+\psi_\infty))} \longrightarrow 0,
\end{split}
\end{equation}

On one hand, 
there exists a positive constant $c$ such that for every
$t\in H^0(Y, k\LL)\otimes_\Z\C$ we have
\[
\|vt\|_{\mu_\infty, k\phi_\infty+\psi_\infty}^2=\int_{\mathcal S} 
\|v(x)\|_{\psi}^2 \| t(x)\|^2_{k\phi}\mu_\infty\geq c 
%\int_{\mathcal S}  \| t(x)\|^2_{k\phi_\infty}\mu_\infty
\|t\|_{\mu_\infty,k\phi_\infty}^2.
\] 
On the other hand, 
\[
\|vt\|_{\mu_\infty, k\phi_\infty+\psi_\infty}^2 \leq % \|v\|_{\sup,\psi}^2 \int_{\mathcal S}  \| t(x)\|^2_{k\phi_\infty}\mu_\infty=
\|v\|_{\sup,\psi_\infty}^2 \|t\|_{\mu_\infty,k\phi_\infty}^2,
\]
for every $t\in H^0(Y, k\LL)_\C$. \\

This leads to the following
\[
c 
\|t\|_{\mathrm{sq},(\mu_\infty,k\phi_\infty)}^2
\leq \|vt\|_{\mathrm{sq},(\mu_\infty, k\phi_\infty+\psi_\infty)}^2 \leq \|v\|_{\sup,\psi}^2 \|t\|_{\mathrm{sq},(\mu_\infty,k\phi_\infty)}^2\quad \forall t \in H^0(Y, k\LL)\otimes_\Z\C.
\]

An easy adaptation of the proof of  \cite[Theorem 4.1]{HajliTAMS} can be used to deduce that % too fast
\[ 
\limsup_{k\rightarrow \infty}\frac{ \hat{h}^0\left(\overline{{H^0}({{X}},k\LL_{|_X} )}_{\mathrm{sq}, (\mu_\infty,k\phi_\infty, {}{v})} \right)}{k^d/d!}=
\limsup_{k\rightarrow \infty}\frac{ \hat{h}^0(\overline{{H^0}({{X}},k\LL_{|_X} )}_{\mathrm{sq}, (\mu_\infty,k\phi_\infty)} )}{k^d/d!}.
\]

From \eqref{XV} and using \cite[(3.3.2), (3.3.3) p. 59 ]{BostTheta} and by \cite[Theorem 4.1]{HajliTAMS}, we infer  that
\[
\limsup_{k\rightarrow \infty} \frac{{{{\hat{h}^0}}} \left(\overline{{H^0}({{X}},(k\LL+\mathcal E)_{|_{{X}}}) }_{{\mathrm{sq}}, (\mu_\infty, k\phi_\infty+\psi_\infty)}\right) }{k^d/d!}= \limsup_{k\rightarrow \infty}\frac{ \hat{h}^0(\overline{{H^0}({{X}},k\LL_{|_X} ))}_{\mathrm{sq}, (\mu_\infty,k\phi_\infty, {}{v})} )}{k^d/d!}.
\]

This concludes the proof of the proposition.

\end{proof}

\begin{theorem}\label{maintheorem} 
We have
\item 
\begin{enumerate}[label=\roman*)]
\item 

\[
\lim_{k\rightarrow \infty} \frac{{{{\hat{h}^0}}} \left(\overline{{H^0}({{X}},(k\LL+\mathcal E)_{|_{{X}}}) }_{{\mathrm{sq}}, (\mu_\infty, k\phi_\infty+\psi_\infty)}\right) }{k^d/d!}=
h_{ \overline{ \LL}_{\phi_\infty}}({{X}}),
\]

\item
\[
\widehat{\mathrm{vol}}_X({\overline{ \LL}_{\phi_\infty}})=h_{ \overline{ \LL}_{\phi_\infty}}({{X}}).
\]

\end{enumerate}
\end{theorem}

\begin{remark}
Chen \cite{Chen1} proved that the limsup in the definition of arithmetic volume is in fact a limit. 
\end{remark}

\begin{proof}

%%%%%%%%

From Theorem \ref{thm4.3}, we get for every $p=1,2,\ldots$
\[
\begin{split}
\limsup_{k\rightarrow \infty} \frac{{ \hat h^0 }(\overline{{H^0}({{X}},(k\LL+ \mathcal{E})_{|_{{X}}}) }_{{\mathrm{sq}}, (\mu, k{{\phi_p}}) }) }{k^d/d!}=&\limsup_{k\rightarrow \infty} \frac{{\hat h^0}(\overline{{H^0}({{X}},(k\LL+ \mathcal{E})_{|_{{X}}}) }_{{\mathrm{sq}}, (\sup, k{{\phi_p}}) }) }{k^d/d!}\\
=&\widehat{\mathrm{vol}}_X(\overline{\LL}_{\phi_p}).
\end{split}
\]

%%%%%
We have 
\[
\begin{split}
\lim_{p\ra \infty}\widehat{\mathrm{vol}}_X(\overline{\LL}_{\phi_p})=&
\lim_{p\rightarrow \infty}  \limsup_{k\rightarrow \infty}\frac{{\hat h^0}(\overline{{H^0}({{X}},{k\LL}_{|_{{X}}}) }_{{\mathrm{sq}},(\mu,k{{\phi_p}})} )}{k^d/d!} \quad \text{(by \eqref{Th2.2ii})}\\
 =&\limsup_{k\rightarrow \infty}\frac{{\hat h^0}(\overline{{H^0}({{X}},{k\LL}_{|_{{X}}}) }_{{\mathrm{sq}},(\mu_\infty,k{{\phi_\infty}})} )}{k^d/d!} \quad \text{(by  \eqref{1706})}\\
=&\limsup_{k\rightarrow \infty} \frac{{{{\hat{h}^0}}} \left(\overline{{H^0}({{X}},(k\LL+\mathcal E)_{|_{{X}}}) }_{{\mathrm{sq}}, (\mu_\infty, k\phi_\infty+\psi_\infty)}\right) }{k^d/d!}\quad\text{(by \eqref{4.3-1})}\\
\end{split}
\]

Hence
\begin{equation}\label{1606}
\limsup_{k\rightarrow \infty}\frac{{\hat h^0}\Bigl(\overline{{H^0}({{X}}, (k\LL+\mathcal E)_{|_{{X}}}) }_{{\mathrm{sq}},(\mu_\infty,k{{\phi_\infty}+\psi_\infty })}\Bigr) }{k^d/d!}=\widehat{\mathrm{vol}}_{{X}}( \overline{\LL}_{\phi_\infty}),
\end{equation}
where we have used that $\lim_{p\rightarrow \infty} \widehat{\mathrm{vol}}_{{X}}( \overline{\LL}_{\phi_p})=\widehat{\mathrm{vol}}_{{X}}( \overline{\LL}_{\phi_\infty}).$\\

Note that
\begin{equation}\label{chiI}
\hat{\chi}(\overline{{H^0}({{Y}},k\LL)}_{(\mu_\infty,k\phi_\infty, {}{s})} )=\det \left(\left<{{s}} {\chi^m},{{s}} \chi^{m'}  \right>_{(\mu_\infty,k\phi_\infty)}  \right)_{m,m'\in {}{k\Delta_\LL\cap M}},
\end{equation}
(we recall that $\left<\cdot,\cdot\right>_{(\mu_\infty,k\phi_\infty)}$ is the  scalar product
associated with $\|\cdot\|_{(\mu_\infty,k \phi_\infty)}$).\\

We have

\begin{equation}\label{deninger}
\lim_{k\rightarrow \infty } \log \left(  \det \Bigl( \int_{\mathbb{S}_N} \chi^{m} \overline{\chi^{m'}} |s|^2 {d\mu_{\infty}} \Bigr)_{m,m'\in  (k\Delta_\LL)\cap M}\right)^{1/|k\Delta_\LL\cap M|} =\mathrm{vol}(\LL_\Q) \int_{\mathcal{S}}\log |s|^2 {d\mu_{\infty}},
\end{equation}
see  \cite[Theorem 4, p.49]{Deninger_Mahler}. \\

We obtain that
\[
\lim_{k\rightarrow \infty}\frac{ \hat{\chi}(\overline{{H^0}({{Y}},k\LL ))}_{(\mu_\infty,k\phi_\infty, {}{s})} )}{k^d/d!}=\mathrm{vol}(\LL_\Q) \int_{\mathcal{S}}\log |s|^2 {d\mu_{\infty}}.
\]

Applying  \cite[p. 81]{Ran} to \eqref{exact}, we obtain  that
\[
\begin{split}
\widehat{\deg}(\overline{{H^0}({{X}},(k\LL+ \mathcal{E})_{|_{{X}}}) })
_{{\mathrm{sq}},(\mu_\infty,k\phi_\infty+\psi_\infty)}=&
\widehat{\deg}(\overline{{H^0}({{Y}},(k\LL+ \mathcal{E}))} )_{(\mu_\infty,k\phi_\infty+\psi_\infty)}\\
&-\widehat{\deg}(\overline{{H^0}({{Y}}, k\LL))})_{(\mu_\infty,k\phi_\infty, {}{s})},
\end{split}
\]
holds for every $k\in\N$.\\

An easy computation shows that
 \[
\hat{\chi}\left( \overline{{H^0}\left({{Y}},(k\LL+ \mathcal{E}) \right)}_{(\mu_\infty,k\phi_\infty+\psi_\infty)} \right)=0\quad \forall \ k\in \N.
\]

Hence, for all $k\in \N$,
\[
\begin{split}
\widehat{\deg}(\overline{{H^0}({{X}},(k\LL+ \mathcal{E})_{|_{{X}}}) })
_{{\mathrm{sq}},(\mu_\infty,k\phi_\infty+\psi_\infty)}=&
\hat{\chi}(\overline{{H^0}({{Y}}, k\LL+ \mathcal E)} _{(\mu_\infty,k\phi_\infty+\psi_\infty)})-\hat{\chi}(\overline{{H^0}({{Y}},k\LL)}_{(\mu_\infty,k\phi_\infty,s)})\\
&-\hat{\chi}(\overline{\Z}^{h^0(Y,(k\LL+ \mathcal{E}))})+\hat{\chi}(\overline{\Z}^{h^0(Y, k\LL)})\\
=& \det \left(\left<s {\chi^m},s \chi^{m'}  \right> _{(\mu_\infty,k\phi_\infty)} \right)_{m,m'\in  k\Delta_\LL\cap M}\\
&-\hat \chi(\overline{\Z}^{\#(\Delta_{k\LL+\mathcal E}\cap M)}) + \hat \chi(\overline{\Z}^{\# (k\Delta_\LL\cap M)}).
\end{split}
\]

 Since $\dim {H^0}\left({{Y}}, k\LL+\mathcal E \right)\otimes_\Z\Q=O(k^d)$ as $k\rightarrow \infty$, we can use Lemma \ref{Gamma} to conclude that
\begin{equation}\label{mahler06}
\lim_{k\rightarrow \infty}\frac{\widehat{\deg}(\overline{{H^0}({{X}},(k\LL+ \mathcal{E})_{|_{{X}}}) }_{ \mathrm{sq}, (\mu_\infty,k\phi_\infty+\psi_\infty)}  )}{k^d/d!}=\mathrm{vol}(\LL_\Q)\int_{\mathcal{S}} \log |s|^2
{d\mu_{\infty}}.
\end{equation}

It is clear that the metric $\|\cdot\|_{\phi_\infty}$ satisfies the conditions of
Lemma \ref{h1}. Using \eqref{h0h1}, we get
\[
\lim_{k\rightarrow \infty}\frac{\widehat{\deg}(\overline{{H^0}({{X}},(k\LL+ \mathcal{E})_{|_{{X}}}) }_{ \mathrm{sq}, (\mu_\infty,k\phi_\infty+\psi_\infty)}  )}{k^d/d!}= \limsup_{k\rightarrow \infty}\frac{{\hat h^0}(\overline{{H^0}({{X}},(k\LL+ \mathcal{E})_{|_{{X}}}) }_{{\mathrm{sq}},(\mu_\infty,k{{\phi_\infty+\psi_\infty}})}}{k^d/d!}.
\]
So, by \eqref{1606} and \eqref{mahler06},
\[
\mathrm{vol}(\LL_\Q)\int_{\mathcal{S}} \log |s|^2
{d\mu_{\infty}}=\widehat{\mathrm{vol}}_{{X}}( \overline{\LL}_{\phi_\infty}).
\]
We conclude that
\[
\widehat{\mathrm{vol}}_{{X}}( \overline{\LL}_{\phi_\infty})=h_{\overline{\LL}_{\phi_\infty}}(X),
\]
where we have used the fact that 
\[h_{\overline{\LL}_{\phi_\infty}}(X)=\mathrm{vol}(\LL_\Q)\int_{\mathcal{S}} \log |s|^2
{d\mu_{\infty}},\] see \cite[Proposition 7.2.1]{Maillot}.
\end{proof}

 %%%%%%

 \section{A generalized Hodge index theorem on hypersurfaces in toric varieties}

We introduced the theory of arithmetic theta invariants  associated with Hermitian line bundles on arithmetic varieties, see \cite[Section 4]{HajliTAMS}.  One of the main result of \cite{HajliTAMS} is 
 a generalized Hodge index theorem on toric varieties. In this section we show that the methods of \cite{HajliTAMS} can be generalized to prove a generalized Hodge index theorem on hypersurfaces in toric varieties.

\begin{theorem}\label{main} Let $Z$ be an arithmetic variety of  dimension $n+1$ and with smooth generic fibre.
Let $(\LL, \|\cdot\|_\phi)$ and 
$(\LL, \|\cdot\|_\psi)$ be two $\mathcal{C}^\infty$ semipositive  Hermitian  line bundles on $Z$. We assume $\psi\leq \phi$ and  $(\LL, \|\cdot\|_\psi)$ is generated by small sections. We have 
\[
\widehat{\mathrm{vol}}(\LL, \|\cdot\|_\phi)-\widehat{\deg}(\hat c_1( \LL, \|\cdot\|_{\phi})^{n+1})
=
\widehat{\mathrm{vol}}(\LL, \|\cdot\|_\psi)-\widehat{\deg}(\hat c_1(\LL, \|\cdot\|_{\psi })^{n+1}).
\]

\end{theorem}
\begin{proof}
See \cite[Theorem 4.11]{HajliTAMS}.
\end{proof}

\begin{theorem}\label{ThmVolDeg} Let $Y$ be a smooth toric variety over $\Z$ of dimension $d+1$. Let $\LL$ be an equivariant line bundle on $Y$. We assume that  it admits a semipositive weight  $\phi$ and  that  $\overline{\LL}_\phi$ is generated by small sections. Let
$X$ be a hypersurface in $Y$.  Then
\[
\widehat{\mathrm{vol}}_X({\overline{ \LL}_{\phi}})=h_{ \overline{ \LL}_{\phi}}({{X}}).
\]
\
\end{theorem}

\begin{proof} The arithmetic volume function is continuous with respect to the variation of the metric, see \cite[p. 513]{Moriwaki}. Then 
we can assume that $\overline{\LL}_\phi$ is generated by strictly small sections. We claim that the proof  can be  obtained by the same method as in the  proof of \cite[Theorem 5.4]{HajliTAMS}. Indeed, 
there exists $0<\alpha\leq 1$ such that
\[
\alpha \|\cdot\|_\phi\leq \|\cdot\|_{\phi_\infty}. %\leq \|\cdot\|_p
\]

So by Theorem \ref{main} and 
 a continuity argument, we can  show that

\[
\widehat{\mathrm{vol}}_X(\LL, \alpha\|\cdot\|_\phi)-\widehat{\deg}(\hat c_1(\LL_{|_X}, \alpha \|\cdot\|_{\psi })^{d}) =
\widehat{\mathrm{vol}}_X(\LL, \|\cdot\|_{\phi_\infty})-\widehat{\deg}(\hat c_1(\LL_{|_X}, \|\cdot\|_{\phi_\infty})^{d}),
\]
and
\[
\widehat{\mathrm{vol}}_X(\LL, \alpha\|\cdot\|_\phi)-\widehat{\deg}(\hat c_1(\LL_{|_X}, \alpha \|\cdot\|_{\phi })^{d})=
\widehat{\mathrm{vol}}_X(\LL, \|\cdot\|_\phi)-\widehat{\deg}(\hat c_1(\LL_{|_X}, \|\cdot\|_{\phi })^{d}).
\]

Since \[
\widehat{\mathrm{vol}}(\overline{\LL}_{\phi_\infty})=h_{\overline{\LL}_{\phi_\infty}}(X),
\]
(see   (ii) of Theorem \ref{maintheorem})  we deduce that
\[
\widehat{\mathrm{vol}}_X(\LL, \|\cdot\|_\phi)=h_{\overline{\LL}_\phi}(X).
\]

 \end{proof}

\begin{theorem}\label{Hodge}[A generalized Hodge index theorem] Let $Y$ be a smooth toric variety over $\Z$ of dimension $d+1$. Let $\LL$ be an equivariant line bundle on $Y$.  Let $\phi$ 
be a semipositive weight on $\LL$. Let
$X$ be a hypersurface in $Y$.  We have
\[
\widehat{\mathrm{vol}}_X(\overline{\LL}_\phi)\geq h_{\overline{\LL}_\phi}(X).
\]
\end{theorem}

\begin{proof} The proof is similar to the proof of \cite[Theorem 5.5]{HajliTAMS}. 
The $\Z$-algebra 
$\bigoplus_{k\geq 0} H^0(X,{k\LL}_{|_X})$ is generated by $\chi^m$ for
$m\in k\Delta_\LL\cap M$.\\

 Let $\alpha$ be a positive real number such that
\[
0<\alpha<1\quad\text{and}\quad \alpha \|\chi^m\|_{\sup,\phi_{|_X}}<1\quad\text{for}\quad m\in k\Delta_\LL\cap M.\]
It follows   that 
 $(\LL,\alpha \|\cdot\|_\phi)$ is ample.\\
 
 From Theorem \ref{ThmVolDeg}, we get
 \[
\widehat{\mathrm{vol}}_X(\LL,  \alpha\|\cdot\|_\phi)=\widehat{\deg}( \hat c_1(\LL, \alpha\|\cdot\|_{\phi_{|_X}} )^{d} ).
\]
 
By \cite[Proposition 4.6 and  Theorem 4.7]{HajliTAMS} we see that
\[
\widehat{\mathrm{vol}}_X(\LL,  \|\cdot\|_\phi) -\widehat{\deg}(\hat c_1(\LL, \|\cdot\|_{\phi_{|_X}})^{d})\geq \widehat{\mathrm{vol}}(\LL,  \alpha\|\cdot\|_{\phi_{|_X}})-\widehat{\deg}(\hat c_1(\LL, \alpha\|\cdot\|_{\phi_{|_X}})^{d}).\]
Therefore
\[
\widehat{\mathrm{vol}}_X(\LL,  \|\cdot\|_\phi)\geq 
\widehat{\deg}(\hat c_1(\LL, \|\cdot\|_{\phi_{|_X}})^{d}).
\]

\end{proof}

\bibliographystyle{plain}

\bibliography{biblio}

\end{document}